\documentclass[smallextended,envcountsect]{svjour3} 
\usepackage{latexsym, amsmath, amsfonts, amscd, amssymb, verbatim, enumerate}
\smartqed 
\usepackage{graphicx}
\usepackage{color,enumerate}

\newtheorem{Definition}{Definition}[section]
\def\R{{\rm I\!R}}
\def\N{{\rm I\!N}}

\begin{document}

\title{Solution Existence Theorems for Finite Horizon Optimal Economic Growth Problems}


\author{Vu Thi Huong}

\institute{Vu Thi Huong \at Institute of Mathematics, Vietnam Academy of Science and Technology,\\ Hanoi, Vietnam\\
vthuong@math.ac.vn, huong263@gmail.com}

\date{{\text{Communicated by ...}}
		\\
		\\
Received: date / Accepted: date}

\maketitle


\smallskip
\begin{abstract}
The solution existence of finite horizon optimal economic growth problems is studied by invoking Filippov's Existence Theorem for optimal control problems with state constraints of the Bolza type from the monograph of L.~Cesari [\textit{Optimization Theory and Applications}, Springer-Verlag, New York, 1983]. Our results are obtained not only for general problems but also for typical ones with the production function and the utility function being either the AK function or the Cobb--Douglas one. Some open questions and conjectures about the regularity of the global solutions of finite horizon optimal economic growth problems are formulated in this paper.
\end{abstract}




\setcounter{equation}{0}
\section{Introduction}
Models of economic growth have played an essential role in economic and mathematical studies since the 30s of the twentieth century. Based on different consumption behavior hypotheses, they allow ones to analyze, plan, and predict relations between global factors, which include \textit{capital}, \textit{labor force}, \textit{production technology}, and \textit{national product}, of a particular economy in a given planning interval of time. Principal models and their basic properties have been investigated by Ramsey \cite{Ramsey_1928}, Harrod \cite{Harrod_1939}, Domar \cite{Domar_1946}, Solow \cite{Solow_1956}, Swan \cite{Swan_1956}, and others. Details about the development of the economic growth theory can be found in the books by Barro and Sala-i-Martin~\cite{Barro_Sala-i-Martin_2004} and Acemoglu~\cite{Acemoglu_2009}. 

 Along with the analysis of the global economic factors, another major issue regarding  an economy is the so-called \textit{optimal economic growth problem}, which can be roughly stated as follows: Define the amount of consumption (and therefore, saving) at each time moment \textit{to maximize a certain target of consumption satisfaction} while fulfilling given relations in the growth model of that economy. This optimal consumption/saving problem was first formulated and solved to a certain extent by Ramsey~\cite{Ramsey_1928}. Later, significant extensions of the model in~\cite{Ramsey_1928} were suggested by ~Cass \cite{Cass_1965} and Koopmans~\cite{Koopmans_1965}.
 
Characterizations of the solutions of optimal economic growth problems (necessary optimality conditions, sufficient optimality conditions, etc.) have been discussed in the books \cite[Chapter~5]{Takayama_1974}, \cite[Chapters~5, 7, 10, and 11]{PierreNVTu_1984}, \cite[Chapter~20]{Chiang_Wainwright_2005}, \cite[Chapters~7 and 8]{Acemoglu_2009}, and some papers cited therein. However, results on the solution existence of these problems seem to be quite rare. For infinite horizon models, some solution existence results were given in \cite[Example~7.4]{Acemoglu_2009} and \cite[Subsection~4.1]{d'Albis_Gourdel_Cuong_2008}. For finite horizon models, our careful searching in the literature leads just to \cite[Theorem~1]{Nikolskii_2016}. This observation  motivates the present investigations.

 This paper considers the solution existence of finite horizon optimal economic growth problems of an aggregative economy; see, e.g., \cite[Sections~C and D in Chapter~5]{Takayama_1974}. It is worthy to stress that we do not assume any special saving behavior, such as \textit{the constancy of the saving rate} as in growth models of Solow \cite{Solow_1956} and Swan \cite{Swan_1956} or the \textit{classical saving behavior} as in \cite[p.~439]{Takayama_1974}. Our main tool is Filippov's Existence Theorem for optimal control problems with state constraints of the Bolza type from the monograph of Cesari~\cite{Cesari_1983}. Our new results on the solution existence are obtained under some mild conditions on the \textit{utility function} and the \textit{per capita production function}, which are two major inputs of the model in question. The results for general problems are also specified for typical ones with the production function and the utility function being either the \textit{AK function} or the \textit{Cobb--Douglas one} (see, e.g., \cite{Barro_Sala-i-Martin_2004} and \cite{Takayama_1974}). Some interesting open questions and conjectures about the \textit{regularity of the global solutions} of finite horizon optimal economic growth problems are formulated in the final part of the paper. Note that, since the saving policy on a compact segment of time would be implementable if it has an infinite number of discontinuities, our concept of regularity of the solutions of the optimal economic growth problem has a clear practical meaning. 
 
The solution existence theorems herein for finite horizon optimal economic growth problems cannot be derived from the above cited results in \cite[Theorem~1]{Nikolskii_2016}, because the assumptions of the latter are more stringent and more complicated than ours. For solution existence theorems in optimal control theory, apart from \cite{Cesari_1983}, the reader is referred to \cite{Lee_Markus_1986}, \cite{Balder_1983}, and the references therein.
 
 The rest of this paper is organized as follows. Section~\ref{Background Materials} presents the modeling of the finite horizon optimal economic growth problems and some background materials including the above-mentioned Filippov's theorem.  Results on the solution existence for general and typical problems are addressed, respectively, in Section~\ref{general-OEG-problems}~and ~\ref{typical-OEG-problems}. Further discussions about the assumptions on the per capita production function and about the regularity of the global solutions are given in Section~\ref{Further discussions}.
 
\section{Preliminaries}\label{Background Materials}
This section collects some notations, definitions, and results that will be used in the sequel. The emphasis will be made on optimal economic growth problems.

 By $\R$ (resp., $\R_+$, and $\N)$ we denote the set of real numbers (resp., the set of nonnegative real numbers, and the set of positive integers). The Euclidean norm in the $n$-dimensional space $\R^n$ is denoted by $\|.\|$. The \textit{Sobolev space}  $W^{1,1}([t_0, T], \R^n)$ (see, e.g., \cite[p.~21]{Ioffe_Tihomirov_1979}) is the linear space of the \textit{absolutely continuous functions} $x~:~[t_0, T] \to \R^n$ equipped with the norm $$\|x\|_{W^{1,1}}=\|x(t_0)\|+\int_{t_0}^T \|\dot x(t)\| dt.$$  It is well-known (see, e.g., \cite{Kolmogorov_Fomin_1970}) that any absolutely continuous function  $x:[t_0, T] \to \R^n$ is Fr\'echet differentiable everywhere on $[t_0, T]$. Moreover, the function $\dot x(\cdot)$ is integrable on $[t_0, T]$ with the integral being understood in the Lebesgue sense. The space $W^{1,1}([t_0, T], \R)$ is vital for us, because the capital-to-labor ratio function $k(\cdot)$ (see Subsection~\ref{OEG-models} below) in the economic growth models is sought in that space.

\subsection{Optimal Economic Growth Models}\label{OEG-models}
Following Takayama \cite[Sections~C and D in Chapter~5]{Takayama_1974}, we consider the problem of \textit{optimal growth of an aggregative economy}.  Suppose that the economy can be characterized by one sector, which produces the \textit{national product} $Y(t)$ at time $t$. Suppose that $Y(t)$ depends on two factors, the \textit{labor} $L(t)$ and the \textit{capital} $K(t)$, and the dependence is described by a \textit{production function} $F$. Namely, one has
\begin{equation*}\label{produ-func}
Y(t)=F(K(t), L(t)),\quad  t\geq 0.
\end{equation*}
It is assumed that $F:\R^2_+\to \R_+$ is a function defined on the nonnegative orthant $\R^2_+$ of $\R^2$ having nonnegative real values, and that it exhibits constant returns to scale, i.e., 
\begin{equation}\label{produ-assum-1}
F(\alpha K, \alpha L)=\alpha F(K, L), \quad \forall \alpha> 0, \, \forall (K, L)\in\R^2_+.
\end{equation}

For every $t\geq 0$, by $C(t)$ and $I(t)$, respectively, we denote the \textit{consumption amount} and the \textit{investment amount} of the economy. The \textit{equilibrium relation} in the output market is depicted by 
 \begin{equation}\label{equil-assum}
 Y(t)=C(t)+I(t), \quad \forall t\geq 0.
 \end{equation} 
The relationship between the capital $K(t)$ and the investment amount $I(t)$ is given by the differential equation
 \begin{equation}\label{invest-assum}
 \dot K(t)=I(t), \quad \forall t\geq 0,
 \end{equation}  where $\dot K(t)=\dfrac{dK(t)}{dt}$ denotes the Fr\'echet derivative of $K(\cdot)$ at time instance $t$ (see, e.g., \cite[pp.~465--466]{Chiang_Wainwright_2005}). If the investment function $I(\cdot)$ is continuous, then one can compute the capital stock $K(t)$ at time $t$ by the formula
 $$K(t)=K(0)+\displaystyle\int_0^t I(\tau)d\tau,$$ where the integral is Riemannian and  $K(0)$ signifies the initial capital stock. In particular, the rate of increase of the capital stock $\dot K(t)$ at every time moment $t$ exists and it is finite. 
 
 If the initial labor amount is $L_0>0$ and the \textit{rate of labor force} is a constant $\sigma>0$ (i.e., $\dot L(t)=\sigma L(t)$ for all $t\geq 0$), then the labor amount at any time $t$ is
\begin{equation}\label{labor-assum}
L(t)=L_0e^{\sigma t}, \quad \forall t\geq 0.
\end{equation}

As $L(t)>0$ for any $t\geq 0$, it follows from \eqref{produ-assum-1} that
$
\frac{Y(t)}{L(t)}=F\Big(\frac{K(t)}{L(t)}, 1\Big)
$ for all $t\geq 0.$
By introducing the \textit{capital-to-labor ratio} $k(t):=\frac{K(t)}{L(t)}$ for $t\geq 0$ and the function $\phi(k):=F(k, 1)$ for $k\geq 0$, from the last equality we have 
\begin{equation}\label{produ-per-cap}
\phi(k(t))=\dfrac{Y(t)}{L(t)}, \quad \forall t\geq 0.
\end{equation}
Due to \eqref{produ-per-cap}, one calls $\phi(k(t))$ the \textit{output per capita} at time $t$ and  $\phi(\cdot)$ the \textit{per capita production function}. Since $F$ has nonnegative values, so does $\phi$. Combining the continuous differentiability of $K(\cdot)$ and $L(\cdot)$, which is guaranteed by \eqref{invest-assum} and \eqref{labor-assum}, with the equality defining the capital-to-labor ratio, one can asserts that $k(\cdot)$ is continuously differentiable. Thus, from the relation $K(t)=k(t)L(t)$ one obtains
$
\dot K(t)=\dot k(t)L(t)+k(t)\dot L(t)
$ for all $t\geq 0.$
Dividing both sides of the last equality by $L(t)$ and recalling that $\dot L(t)=\sigma L(t)$, we get
\begin{equation}\label{dot K/L-1}
\dfrac{\dot K(t)}{L(t)}=\dot k(t)+\sigma k(t), \quad \forall t\geq 0.
\end{equation}
Similarly, dividing both sides of the equality in \eqref{invest-assum} by $L(t)$ and using \eqref{equil-assum}, we have
$
\frac{\dot K(t)}{L(t)}=\frac{Y(t)}{L(t)}-\frac{C(t)}{L(t)}
$ for all $t\geq 0.$
So, by considering the \textit{per capita consumption} $c(t):=\frac{C(t)}{L(t)}$ of the economy at time $t$ and invoking \eqref{produ-per-cap}, one obtains
$
\frac{\dot K(t)}{L(t)}=\phi(k(t))-c(t)
$ for all $t\geq 0.$
Combining this with \eqref{dot K/L-1} yields
\begin{equation}\label{state-k-control-c}
\dot k(t)=\phi(k(t))-\sigma k(t)-c(t), \quad \forall t\geq 0.
\end{equation} 
The amount of consumption at time $t$ is
\begin{equation}\label{consu-assum}
C(t)=(1-s(t))Y(t),\quad \forall t\geq 0,
\end{equation}
with $s(t) \in [0, 1]$ being the \textit{propensity to save} at time $t$ (thus, $1 - s(t)$ is the \textit{propensity to consume} at time $t$). Then, by dividing both sides of \eqref{consu-assum} by $L(t)$ and referring to \eqref{produ-per-cap}, one gets
\begin{equation}\label{consu-savi-relation}
c(t)=(1-s(t))\phi(k(t)),\quad \forall t\geq 0.
\end{equation}
Thanks to \eqref{consu-savi-relation}, one can rewrite \eqref{state-k-control-c} equivalently as
\begin{equation}\label{state-k-control-s}
\dot k(t)=s(t)\phi( k(t))-\sigma k(t), \quad \forall t\geq 0.
\end{equation}
In the special case where $s(\cdot)$ is a constant function, i.e., $s(t)=s>0$ for all $t\geq 0$, the relation~\eqref{state-k-control-s} is the fundamental equation of the \textit{neo-classical aggregate growth model} of Solow \cite{Solow_1956}.

One major concern of the planners is to choose a pair of functions $(k, c)$ (or $(k, s)$) defined on a planning interval $[t_0, T]\subset [0, +\infty]$, that satisfies \eqref{state-k-control-c} (or \eqref{state-k-control-s}) and the initial condition $k(t_0)=k_0$, to maximize a certain target of consumption. Here $k_0 > 0$ is a given value. As the  target function one may  choose is $\int_{t_0}^T c(t)dt$, which is the total amount of per capita consumption on the time period $[t_0, T]$. A more general kind of the target function is $\int_{t_0}^T \omega (c(t))e^{-\lambda t}dt$, where $\omega :\R_+ \to \R$ is a \textit{utility function} associated with the representative individual consumption $c(t)$ in the society, $e^{-\lambda t}$ is the \textit{time discount factor}, and  $\lambda \geq 0$ is the \textit{real interest rate}. Clearly, the former target function is a particular case of the latter one with $\omega(c)=c$ being a linear utility function and the real interest rate $\lambda=0$. For more discussions about the length of the planning interval, the choice the utility function $\omega(\cdot)$ (\textit{it must be linear, or it can be nonlinear?}), as well the choice of the real interest rate (\textit{one must have $\lambda=0$, or one can have $\lambda > 0$?}), we refer the reader to \cite[pp.~445--447]{Takayama_1974}. 

The just mentioned planning task is an \textit{optimal control problem}. Interpreting $k(t)$ as the state trajectory and $s(t)$ as the control function, we can formulate the problem as follows. 

 Let there be given a production function $F:\R^2_+\to \R_+$ satisfying \eqref{produ-assum-1} for any $(K, L)$ from $\R^2_+$ and $\alpha> 0$. Define the function $\phi(k)$ on $\R_+$ by setting $\phi(k)=F(k, 1)$. Assume that a finite time interval $[t_0, T]$ with  $T>t_0\geq 0$, a utility function $\omega: \R_+\to \R$, and a real interest rate $\lambda\geq 0$ are given. Since $c(t)=(1-s(t))\phi(k(t))$ by \eqref{consu-savi-relation}, the target function can be expressed via $k(t)$ and $s(t)$ as $\int_{t_0}^T \omega (c(t))e^{-\lambda t}dt=\int_{t_0}^{T} \omega[(1-s(t))\phi (k(t))]e^{-\lambda t} dt.$  So, the problem of finding an optimal growth process for an aggregative economy is the following one:
\begin{equation} \label{cost functional_GP}
\mbox{Maximize}\ \; I(k,s):=\int_{t_0}^{T} \omega[(1-s(t))\phi (k(t))]e^{-\lambda t} dt
\end{equation}
over $k \in W^{1,1}([t_0, T], \R)$  and measurable functions $s:[t_0, T] \to \R$ satisfying
\begin{equation}\label{state control system_GP}
\begin{cases}
\dot k(t)=s(t)\phi (k(t))-\sigma k(t),\quad &\mbox{a.e.\ } t\in [t_0, T]\\
k(t_0)=k_0\\
s(t)\in [0,1], &\mbox{a.e.\ } t\in [t_0, T]\\
k(t) \geq 0, & \forall t\in [t_0, T].
\end{cases}
\end{equation}
This problem has five parameters: $T > t_0 \geq 0,\, \lambda \geq 0,\, \sigma>0$, and  $k_0 > 0$.

The optimal control problem in \eqref{cost functional_GP}--\eqref{state control system_GP} will be denoted by $(GP)$. This is a finite horizon optimal control problem of the Lagrange type with state constraints.

To make $(GP)$ competent with the given modeling presentation, one has to explain why the state trajectory can be sought in $W^{1,1}([t_0, T], \R)$ and the control function is just required to be measurable. If one assumes that the investment function $I(\cdot)$ is continuous on $[t_0, T]$, then~\eqref{invest-assum} implies that $K(\cdot)$ is continuously differentiable; hence so is $k(\cdot)$. However, in practice, the investment function $I(\cdot)$ can be discontinuous at some points $t\in [t_0, T]$ (say, the policy has a great change, and the government decides to allocate a large amount of money into the production field, or to cancel a large amount of money from it). Thus, the requirement that $k(\cdot)$ is differentiable at these points may not be fulfilled. To deal with this situation, it is reasonable to assume that the state trajectory $k(\cdot)$ belongs to the space of continuous, piecewise continuously differentiable functions on $[t_0, T]$, which is endowed with the norm $\|k\|=\displaystyle\max_{t\in [t_0,T]}|k(t)|$.  Since the latter space is incomplete one embeds it into the space $W^{1,1}([t_0, T], \R)$, which possesses many good properties (see \cite{Kolmogorov_Fomin_1970}). In that way, tools from the Lebesgue integration theory and results from the conventional optimal control theory can be used for $(GP)$. Now, concerning the control function $s(\cdot)$, one has the following observation. Since the derivative $\dot k(t)$ exists almost everywhere on $[t_0,T]$ and  $\dot k(\cdot)$ is a measurable function, for the fulfillment of the relation $\dot k(t)=s(t)\phi (k(t))-\sigma k(t)$ almost everywhere on $[t_0, T]$, it suffices to assume that $s(\cdot)$ is a measurable function. Recall that a function $\varphi:[t_0, T]\to\R$ is said to be \textit{measurable} if for any $\alpha\in\R$ the set $\{t\in [t_0, T]\,:\,\varphi\in (-\infty,\alpha)\}$ is Lebesgue measurable. 

\subsection{Filippov's Existence Theorem for State Constrained Bolza Problems}
To recall a solution existence theorem for finite horizon optimal control problems with state constraints of the Bolza type, we will use the notations and concepts given in the monograph of Cesari \cite[Sections~9.2, 9.3, and 9.5]{Cesari_1983}. Let $A \subset\R\times \R^n$ and $U: A\rightrightarrows \R^m$ be a set-valued map defined on $A$. Let $$M:=\{(t, x, u)\in \R\times\R^n\times \R^m \;:\; (t, x)\in A,\ u \in U(t, x)\},$$ $f_0(t, x, u)$ and $f(t, x, u)=(f_1, f_2, \dots, f_n)$ be functions defined on $M$. Let $B$ be a given subset of $\R\times \R^n\times\R\times \R^n$ and $g(t_1, x_1, t_2, x_2)$ be a real valued function defined on $B$. Let there be given an interval $[t_0, T]\subset\R$. Consider the problem of minimizing the function
\begin{equation}\label{cost functional_SET}
I(x, u):= g(t_0, x(t_0), T,  x(T))+\int_{t_0}^{T} f_0(t, x(t), u(t))dt
\end{equation} over pairs of functions $(x, u)$ such that $x (\cdot):[t_0, T]~\to~\R^n$ is absolutely continuous, $u(\cdot):[t_0, T]~\to~\R^m$  is measurable, $f_0(., x(\cdot), u(\cdot)): [t_0, T]~\to~\R$ is Lebesgue integrable, and 
\begin{equation}\label{state control system_SET}
	\begin{cases}
		\dot x(t)=f(t, x(t), u(t)),\quad &\mbox{a.e.\ } t\in [t_0, T]\\
		u(t)\in U(t, x(t)), &\mbox{a.e.\ } t\in [t_0, T]\\
		(t, x(t))\in A, & \forall t\in [t_0, T]\\
		(t_0, x(t_0), T, x(T))\in B.\\		
	\end{cases}
\end{equation} Such a pair $(x, u)$ is called a \textit{feasible process}. The problem \eqref{cost functional_SET}--\eqref{state control system_SET}, which is an optimal control of the Bolza type with state constraints, is denoted by $\mathcal{B}$.

If  $(x, u)$ is a feasible process for $\mathcal B$, then $x$ is said to be a \textit{feasible trajectory}, and $u$ a \textit{feasible control}. The set of all the feasible processes for $\mathcal B$ is denoted by $\Omega$.	A feasible process $(\bar x, \bar u)$ is said to be a \textit{global minimizer} for  $\mathcal B$ if one has $I(\bar x, \bar u)\leq I(x, u)$ for any feasible process $(x, u)$.

Let $A_0:=\big\{t\,:\, \exists x\in \mathbb R^n\ {\rm s.t.}\ (t,x)\in A\big\}$. Set $$A(t)=\big\{x\in \R^n \;:\; (t, x)\in A\big\}$$ for each $t\in A_0$ and  $$\widetilde Q(t, x)=\big\{(z^0, z)\in \R^{n+1} : z^0\geq f_0(t, x, u),\ z=f(t, x, u) \ \mbox{for some}\ u \in U(t, x)\big\}$$ for every $(t, x)\in A$.

The forthcoming statement is known as \textit{Filippov's Existence Theorem for Bolza problems}.
\begin{theorem}[{see \cite[Theorem~ 9.3.i, p.~317, and Section~9.5]{Cesari_1983}}]\label{Filippov's_Existence_Bolza}
Suppose that $\Omega$ is nonempty, $B$ is closed, $g$ is lower semicontinuous on $B$, $f_0$ and $f$ is continuous on $M$ and, for almost every $t\in [t_0, T]$, the sets $\widetilde Q(t, x)$, $x\in A(t)$, are convex. Moreover,  assume either that $A$ and $M$ are compact or that $A$ is not compact but closed and contained in a slab $[t_1, t_2]\times \R^n$ with $t_1$ and $t_2$ being finite, and the following conditions are fulfilled:
\begin{enumerate}[\rm (a)]
\item For any $\varepsilon\geq 0$, the set $M_\varepsilon:=\{(t, x, u)\in M \;:\; \|x\| \leq \varepsilon\}$ is compact;
\item There is a compact subset $P$ of $A$ such that every feasible trajectory $x$ of $\mathcal B$ passes through at least one point of $P$;
\item There exists $c\geq 0$ such that 
$x_1 f_1(t, x, u)+\dots+x_n f_n(t, x, u) \leq c (\|x\|^2+1)$ for all $ (t, x, u)\in M.$
\end{enumerate}
Then, $\mathcal B$ has a global minimizer.
\end{theorem}
Clearly, condition (b) is satisfied if the initial point $(t_0, x(t_0))$ or the end point $(T, x(T))$ is fixed. As shown in \cite[p.~317]{Cesari_1983}, the following condition implies (c): 
\begin{enumerate}[(${\rm c_0}$)]
\item \textit{There exists $c\geq 0$ such that $
\|f(t, x, u)\| \leq c (\|x\|+1)$ for all $(t, x, u)\in M$.}
\end{enumerate}

In the next two sections, several results on the solution existence of optimal economic growth problems will be derived from Theorem~\ref{Filippov's_Existence_Bolza}. 
	  		  	
\section{General Optimal Economic Growth Problems}\label{general-OEG-problems}

Our first result on the solution existence of the finite horizon optimal economic growth problem~$(GP)$ in~\eqref{cost functional_GP}--\eqref{state control system_GP} is stated as follows.

\begin{theorem}\label{Existence_Thm1} For the problem $(GP)$, suppose that $\omega(\cdot)$ and $\phi(\cdot)$ are continuous on $\R_+$. If, in addition,  $\omega(\cdot)$ is concave on $\R_+$ and the function $\phi(\cdot)$ satisfies the condition
\begin{itemize}
\item[${\rm (c_1)}$] There exists $c\geq 0$ such that $\phi (k) \leq (c-\sigma)k+c$ for all  $k\in \R_+$,
\end{itemize}
then $(GP)$ has a global solution.
\end{theorem}
\begin{proof} To apply Theorem~\ref{Filippov's_Existence_Bolza}, we have to interpret $(GP)$ in the form of $\mathcal B$. For doing so,  
we let the variable $k$ (resp., the variable $s$) play the role of the phase variable $x$ in $\mathcal B$ (resp., the control variable $u$ in $\mathcal B$). Then, $(GP)$ has the form of $\mathcal B$ with $n=m=1$, $A=[t_0, T]\times \R_+$, $U(t, k)=[0, 1]$ for all $(t, k)\in A$, $B=\{t_0\}\times \{k_0\}\times \{T\}\times \R$, $M=[t_0, T]\times \R_+\times [0, 1]$, $g\equiv 0$ on $B$, $f_0(t, k, s)=-\omega ((1-s)\phi(k))e^{-\lambda t}$, and $f(t, k, s)=s\phi(k)-\sigma k$ for all $(t, k, s) \in M$. 

Setting $s(t)=0$ and $k(t)=k_0e^{-\sigma(t-t_0)}$ for all $t\in [t_0, T]$, one can easily verify that the pair $(k(\cdot), s(\cdot))$ is a feasible process for $(GP)$. Thus, the set $\Omega$ of the feasible processes is nonempty. It is clear that $B$ is closed, $g$ is continuous on $B$ and, by the assumed continuity of  $\omega(\cdot)$ and $\phi(\cdot)$, $f_0$ and $f$ are continuous on $M$. Besides, the formula for $A$ implies that  $A_0=[t_0, T]$ and $A(t)=\R_+$ for all $t\in A_0$. In addition, by the formulas for $f_0$, $f$ and $U$, one has for any $(t, k)\in A$ the following:
\begin{align*}
&\widetilde Q(t, k)\\
&=\big\{(z^0, z)\in \R^2 : z^0\geq f_0(t, k, s),\  z=f(t, k, s)\ \, {\rm for\ some}\ s \in U(t, k)\big\}\\
&=\big\{(z^0, z)\in \R^2  : \exists s \in [0, 1]\ \, {\rm s.t.}\ z^0\geq -\omega ((1-s)\phi(k))e^{-\lambda t},\ z=s\phi(k)-\sigma k\big\}.
\end{align*} Let us show that, for any  $t\in [t_0, T]$ and $k\in A(t)=\R_+$, the set $\widetilde Q(t, k)$ is convex. Indeed, given any $(z^0_1, z_1)$, $(z^0_2, z_2) \in \widetilde Q(t, k)$ and $\mu \in [0, 1]$, one can find $s_1, s_2 \in [0,1]$ such that 
\begin{align*}
z^0_1\geq -\omega ((1-s_1)\phi(k))e^{-\lambda t},\ \, z_1=s_1\phi(k)-\sigma k,\\
z^0_2\geq -\omega ((1-s_2)\phi(k))e^{-\lambda t},\ \, z_2=s_2\phi(k)-\sigma k.
\end{align*}
Therefore, it holds that
\begin{equation}\label{ineq_1}
\mu z^0_1+(1-\mu)z^0_2\geq  -\mu\omega ((1-s_1)\phi(k))e^{-\lambda t}-(1-\mu)\omega ((1-s_2)\phi(k))e^{-\lambda t}
\end{equation}
and 
\begin{equation}\label{eq_1}
\mu z_1+(1-\mu)z_2= \mu [s_1\phi(k)-\sigma k]+(1-\mu)[s_2\phi(k)-\sigma k].
\end{equation} Setting $s_\mu=\mu s_1+(1-\mu) s_2$, one has $s_\mu\in [0,1]$ and it follows from \eqref{eq_1} that
\begin{equation}\label{eq_2}
\mu z_1+(1-\mu)z_2=s_\mu\phi(k)-\sigma k.
\end{equation}  Clearly, the concavity of $\omega(\cdot)$ on $\R_+$ yields
\begin{align*}
 &-\mu\omega ((1-s_1)\phi(k))-(1-\mu)\omega ((1-s_2)\phi(k))\\
 &\geq-\omega [\mu(1-s_1)\phi(k)+(1-\mu)(1-s_2)\phi(k)] = -\omega ((1-s_\mu)\phi(k)).
\end{align*}
Hence, by \eqref{ineq_1} we obtain
$
\mu z^0_1+(1-\mu)z^0_2\geq -\omega[{(1-s_\mu)\phi(k)]e^{-\lambda t}},
$
which together with \eqref{eq_2} implies that $\mu(z^0_1, z_1)+(1-\mu)(z^0_2, z_2) \in \widetilde Q(t, k)$.

Now, although  $A=[t_0, T]\times \R_+$ is noncompact,  the fact that $A$ is closed and contained in a slab $[t_1, t_2]\times \R$ with $t_1$ and $t_2$ being finite is clear. It remains to check the conditions (a)--(c) in Theorem~\ref{Filippov's_Existence_Bolza}.

For any $\varepsilon\geq 0$, the set $M_\varepsilon$ is compact because
\begin{align*}
M_\varepsilon&=\{(t, k, s)\in[t_0, T]\times \R_+\times [0,1] \;:\; |k| \leq \varepsilon\}\\
&= [t_0, T]\times [0, \varepsilon]\times [0,1].
\end{align*}
So, condition (a) is satisfied. As $P:=\{(t_0,k_0)\}$ is a compact subset of $A$, and every feasible trajectory of $(GP)$ passes through $(t_0,k_0)$, condition (b) is fulfilled. Applied to the case of $(GP)$, where $f(t, k, s)=s\phi(k)-\sigma k$ and $M=[t_0, T]\times \R_+\times [0, 1]$ as explained above, condition (c) in Theorem~\ref{Filippov's_Existence_Bolza} can be rewritten as
 \begin{enumerate}
 	\item[${\rm (c')}$] \textit{There exists  $c\geq0$ such that $\, sk\phi(k)\leq (c+\sigma)k^2+c\,$ for all $(k, s)$ in $\R_+\times[0,1]$.}
 \end{enumerate}
By the comment given after Theorem~\ref{Filippov's_Existence_Bolza}, condition (c) is valid if  condition (${\rm c_0}$) holds. As $f(t, k, s)=s\phi(k)-\sigma k$ and $M=[t_0, T]\times\R\times[0,1]$, the latter can be stated as 
\begin{enumerate}
\item[${\rm (c_0')}$] \textit{There exists  $c\geq0$ such that $|s\phi (k)-\sigma k| \leq c(k+1)$ for all $(k, s)$ in $\R_+\times[0,1]$.}
\end{enumerate}
To prove ${\rm (c_0')}$, observe that the estimates 
\begin{align}\label{1st-hand}
|s\phi (k)-\sigma k| \leq s\phi (k)+\sigma k \leq \phi (k)+\sigma k
\end{align} hold for any $(k, s)\in \R_+\times[0,1]$. Furthermore, thanks to the assumption ${\rm (c_1)}$, we can find a constant $c\geq 0$ such that $\phi (k) \leq (c-\sigma)k+c$ for all $k\in \R_+.$ Since the last inequality can be rewritten as $\phi (k)+\sigma k\leq c(k+1)$, from~\eqref{1st-hand} we get ${\rm (c_0')}$.

Since our problem $(GP)$ in the interpretation given above satisfies all the assumptions of Theorem~\ref{Filippov's_Existence_Bolza}, we conclude that it has a global solution. 
\end{proof}	 

In Theorem \ref{Existence_Thm1}, it is not required that $\phi(\cdot)$ is concave on $\R_+$. It turns out that if the concavity of $\phi(\cdot)$ is available, then there is no need to check ${\rm (c_1)}$. Since the assumption saying that the per capita production function $\phi(k):=F(k, 1)$ is concave on $\R_+$ is reasonable in practice, next theorem seems to be interesting. 

\begin{theorem}\label{Existence_Thm2} If both functions $\omega(\cdot)$ and $\phi(\cdot)$ are continuous and concave on $\R_+$, then $(GP)$ has a global solution.
\end{theorem}

\begin{proof} Set $\psi=-\phi$ and put $\psi(k)=+\infty$ for every $k\in (-\infty,0)$. Then, the function $\psi:\R\to \R\cup\{+\infty\}$ is a proper convex function and the effective domain ${\rm dom}\,\psi$ of $\psi$ is $\R_+$. Select any $\bar k>0$. Since $\bar k$ belongs to the interior of ${\rm dom}\,\psi$, by \cite[Theorem~23.4]{Rockafellar_1970} we know that the \textit{subdifferential} (see, e.g., \cite[p.~215]{Rockafellar_1970}) $\partial\psi(\bar k)$ of $\psi$ at $\bar k$ is nonempty. Thus, taking an element $a\in\partial\psi(\bar k)$, one has
\begin{equation*}
\psi(k)-\psi(\bar k)\geq a(k-\bar k), \quad \forall k\geq 0,
\end{equation*}
or, equivalently,
\begin{equation}\label{subgrad_ineq}
\phi(k)\leq -ak+a\bar k+\phi(\bar k), \quad \forall k\geq 0.
\end{equation}
For $c:=\max\{0, \sigma-a, \phi(\bar k)+a\bar k\}$, one has $c\geq 0$ and
\begin{equation}\label{linear_ineq}
-ak+a\bar k+\phi(\bar k)\leq (c-\sigma)k+c, \quad \forall k\geq 0.
\end{equation}
Combining \eqref{subgrad_ineq} and \eqref{linear_ineq}, one can assert that condition  ${\rm (c_1)}$ in Theorem~\ref{Existence_Thm1} is fulfilled. Thus, the assumed continuity of  $\omega(\cdot)$ and $\phi(\cdot)$ together with the concavity of  $\omega(\cdot)$ allows us to apply Theorem~\ref{Existence_Thm1} to conclude that $(GP)$ has a global solution.
\end{proof}

The next proposition reveals the nature of condition ${\rm (c_1)}$, which is essential for the validity of Theorem~\ref{Existence_Thm1}.  

\begin{proposition}\label{equiv-conditions}
Condition ${\rm (c_1)}$ and the conditions ${\rm (c')}$ and ${\rm (c_0')}$, which were formulated in the proof of Theorem~\ref{Existence_Thm1}, are equivalent. Moreover, each of these conditions is equivalent to the condition 
\begin{equation}\label{c_2}
\limsup_{k\to +\infty}\dfrac {\phi(k)}{k}<+\infty
\end{equation} on the asymptotic behavior of $\phi$.
\end{proposition}
\begin{proof}
The implications ${\rm (c_1)}\Rightarrow {\rm (c_0')}$ and ${\rm (c_0')}\Rightarrow {\rm (c')}$ were obtained in the proof of Theorem~\ref{Existence_Thm1}. So, the proposition will be proved if we can show that ${\rm (c')}$ implies \eqref{c_2} and \eqref{c_2} implies ${\rm (c_1)}$.

To get the implication $ {\rm (c')}\Rightarrow $ \eqref{c_2}, suppose that ${\rm (c')}$  holds. Then, there exists $c\geq 0$ satisfying $\,sk\phi(k)\leq (c+\sigma)k^2+c\,$ for all $\,(k, s)\in \R_+\times[0,1]$. Thus, choosing $s=1$, one has 
$$\dfrac{\phi(k)}{k}\leq c+\sigma+\dfrac{c}{k^2},\quad \forall k>0.$$ By taking the limsup on both sides of the last inequality when $k\to+\infty$, one gets \eqref{c_2}.

Now, to obtain the implication \eqref{c_2} $\Rightarrow {\rm (c')}$, suppose that \eqref{c_2} holds. Then, there exist $\gamma_1>0$ and $\mu>0$ such that $\frac{\phi(k)}{k}\leq \gamma_1$ for every $k>\mu$. Thanks to the continuity of $\phi$ at $k=0$, one can find $\gamma_2>0$ and $\varepsilon\in (0,\mu)$ such that $\phi(k)\leq \gamma_2$ for all $k\in [0,\varepsilon)$. Moreover, by the continuity of the function $k\mapsto \phi(k)/k$ on the compact interval $[\varepsilon,\mu]$, the number $\gamma_3:=\max\Big\{\frac{\phi(k)}{k}\;:\; k\in [\varepsilon,\mu]\Big\}$ is well defined. Thus, for any $c\geq \max\{\gamma_1+\sigma,\,\gamma_2,\,\gamma_3+\sigma\}$, it holds that
$$\phi (k)\leq\gamma_2\leq c\leq (c-\sigma)k+c\quad \forall k\in [0,\varepsilon),$$
$$\phi (k)\leq \gamma_3k\leq (c-\sigma)k+c\quad \forall k\in [\varepsilon,\mu]$$ and
$$\phi (k)\leq \gamma_1k\leq (c-\sigma)k+c\quad \forall k\in (\mu,+\infty).$$ Therefore, one has $\,\phi (k)\leq (c-\sigma)k+c\,$ for every $\,k\geq 0$, which justifies ${\rm (c_1)}$. \end{proof}

\begin{remark}\label{Rem1} {\rm
		There are many continuous functions $\phi: \R_+\to \R_+$ that are nonconcave on $\R_+$  but satisfy condition ${\rm (c_1)}$ in Theorem~\ref{Existence_Thm1}. Indeed, suppose that the values $\bar k>0$, $\phi_0\geq 0$, and $a>0$ are given arbitrarily. Setting
		\begin{equation*}
		\phi(k)=
		\begin{cases}
		\phi_0, \quad &\mbox{if}\ \, k \in [0,\bar k]\\
		a(k-\bar k)+\phi_0, & \mbox{if}\ \, k \in (\bar k, -\infty),
		\end{cases}
		\end{equation*} 
		one has a function $\phi$, that is continuous and nonconcave on $\R_+$. But, since the coercivity condition \eqref{c_2} is fulfilled, this $\phi$ satisfies ${\rm (c_1)}$. More generally, the continuous function 
		\begin{equation*}
		\phi(k)=
		\begin{cases}
		\phi_1(k), \quad &\mbox{if}\ \, k \in [0,\bar k]\\
		a(k-\bar k)^\alpha+\phi_1(\bar k), & \mbox{if}\ \, k \in (\bar k, -\infty),
		\end{cases}
		\end{equation*} where $\alpha\in (0,1]$ is a constant and  $\phi: [0,\bar k]\to \R_+$ is a continuous function, also satisfies ${\rm (c_1)}$ because  \eqref{c_2} is fulfilled. Clearly, there are many ways to choose $\phi_1(k)$ such that this function $\phi$ in nonconcave on $\R_+$.}
\end{remark}

Economic growth problems with utility functions $\omega(\cdot)$ and production functions $F(\cdot)$ of two typical types will be the subject of our consideration in next section.

\section{Typical Optimal Economic Growth Problems}\label{typical-OEG-problems}
As observed by Takayama \cite[p. 450]{Takayama_1974}, the production function given by \begin{equation}\label{const-cap-out-ratio-F}
F(K, L)=\dfrac{1}{a}K, \quad \forall(K, L)\in \R_+^2,
\end{equation}
where $a>0$ is a constant representing the \textit{capital-to-output ratio}, is of a great importance. This function is in the form of the \textit{$AK$ function} (see, e.g., \cite[Subsection~1.3.2]{Barro_Sala-i-Martin_2004}) with \textit{the diminishing returns to capital being absent}, which is a key property of endogenous growth models. The function in \eqref{const-cap-out-ratio-F} is also referred to in connection with the Harrod-Domar model of which a main assumption is that the labor factor is not explicitly involved in the production function (see, e.g., \cite[Footnote~5, p. 464]{Takayama_1974}). In the notations of Subsection~\ref{OEG-models}, by  \eqref{const-cap-out-ratio-F} one has
\begin{equation*}\label{const-cap-out-ratio-phi}
\phi(k)=\dfrac{1}{a}k, \quad \forall k\geq 0.
\end{equation*}
So, the differential equation in \eqref{state control system_GP} becomes
\begin{equation*}\label{const-cap-out-ratio-dot k}
\dot k(t)=\dfrac{1}{a}s(t)k(t)-\sigma k(t),\quad \mbox{a.e.\ } t\in [t_0, T].
\end{equation*}

Another important type of the production function $F$ is the \textit{Cobb-Douglas function} (see, e.g., \cite[p.~29]{Barro_Sala-i-Martin_2004}), which is given by 
\begin{equation}\label{Cobb-Douglas-F}
F(K, L)=A K^\alpha L^{1-\alpha},\quad \forall(K, L)\in \R_+^2,
\end{equation}
with $A > 0$ and $\alpha \in (0, 1)$ being constants. The exponent $\alpha$ (resp., $1-\alpha$) refers to the \textit{output elasticity of capital} (resp., the \textit{output elasticity of labor}), which represents the share of the contribution of the capital (resp., of the labor) to the total product $F(K, L)$. Meanwhile, $A$ expresses the \textit{total factor productivity} (TFP; see, e.g., https://en.wikipedia.org/wiki/Total$\_$factor$\_$productivity). This measure of economic efficiency is the ratio of output over the weighted average of labor and capital input. TFP represents the increase in total production which is in excess of the increase that results from increase in inputs and depends on some intangible factors such as technological change, education, research and development, etc. As $\alpha \in (0, 1)$ , $F$ \textit{exhibits diminishing returns to capital and labor} (see, e.g., \cite[p.~433]{Takayama_1974}). The latter means that the \textit{marginal products} of both capital and labor are diminishing (see, e.g.,\cite[p.~29]{Acemoglu_2009}). The presence of diminishing returns to capital, which plays a very important role in many results of the basic growth model (see, e.g., \cite[p.~29]{Acemoglu_2009}), distinguishes the production given by \eqref{Cobb-Douglas-F} with the one in \eqref{const-cap-out-ratio-F}. The per capita production function corresponding to \eqref{Cobb-Douglas-F} is
\begin{equation}\label{Cobb-Douglas-phi}
\phi(k)=A k^\alpha, \quad \forall k\geq 0.
\end{equation}
Therefore, \eqref{state control system_GP} collapses to
\begin{equation}\label{Cobb-Douglas-dot k}
\dot k(t)=As(t)k^\alpha(t)-\sigma k(t),\quad  \mbox{a.e.\ } t\in [t_0, T].
\end{equation}

Since \eqref{const-cap-out-ratio-F} can be written in the form of \eqref{Cobb-Douglas-F} with $\alpha:=1$ and $A:=1/a$, one can combine the above two types of production functions in a general one by considering \eqref{Cobb-Douglas-F} with $A>0$ and  $\alpha \in (0, 1]$. This means that one has deal with the model \eqref{Cobb-Douglas-phi}--\eqref{Cobb-Douglas-dot k}, where $A>0$ and $\alpha \in (0, 1]$ are given constants. In the same manner, concerning the utility function $\omega(\cdot)$, the formula 
\begin{equation}\label{Cobb-Douglas-omega}
\omega (c)=c^\beta, \quad \forall c\geq 0
\end{equation}
with $\beta \in (0, 1]$ can be considered. For $\beta =1$, $\omega(\cdot)$ is a linear function. For $\beta \in (0, 1)$, it is a Cobb-Douglas function.

In the rest of this section, for the problem $(GP)$, we assume that $\phi(\cdot)$ and $\omega(\cdot)$ are given respectively by \eqref{Cobb-Douglas-phi} and \eqref{Cobb-Douglas-omega}. Then, the target function of $(GP)$ is 
\begin{equation*}
I(k,s)=\int_{t_0}^{T} [1-s(t)]^\beta\phi^\beta(k(t)) e^{-\lambda t}dt=A^\beta\int_{t_0}^{T} [1-s(t)]^\beta k^{\alpha\beta}(t)  e^{-\lambda t}dt.
\end{equation*}
Thus, we have to solve the following equivalent problem:
		\begin{equation} \label{cost functional_GP_1}
		\mbox{Maximize}\ \; \int_{t_0}^{T} [1-s(t)]^\beta k^{\alpha\beta}(t)  e^{-\lambda t}dt
		\end{equation}
		over $k \in W^{1,1}([t_0, T], \R)$  and measurable functions $s:[t_0, T] \to\R$ satisfying
		\begin{equation} \label{state control system_GP_1}
		\begin{cases}
		\dot k(t)=Ak^\alpha(t)s(t)-\sigma k(t),\quad &\mbox{a.e.\ } t\in [t_0, T]\\
		k(t_0)=k_0\\
		s(t)\in [0,1], &\mbox{a.e.\ } t\in [t_0, T]\\
		k(t)\in [0, +\infty), & \forall\ t\in [t_0, T]
		\end{cases}
		\end{equation}
		with $\alpha\in (0, 1]$, $\beta \in (0, 1]$, $A>0$, $T>t_0\geq 0$, $\lambda\geq 0$,  $\sigma>0$, and  $k_0 \geq 0$ being given parameters.
		
		The forthcoming result is a consequence of Theorem~\ref{Existence_Thm2}.
 
 \begin{theorem}\label{Existence_Thm3}
 For any $\alpha\in (0, 1]$ and $\beta \in (0, 1]$, the optimal economic growth problem in \eqref{cost functional_GP_1}--\eqref{state control system_GP_1} possesses a global solution.
 \end{theorem}
 \begin{proof} By the assumptions $A>0$,  $\alpha\in (0, 1]$, and $\beta \in (0, 1]$,  the functions $\phi(k)=Ak^\alpha$ and $\omega(c)=c^{\beta}$ are continuous on $\R_+$. The concavity of $\phi(\cdot)$ on $(0,+\infty)$ follows from the fact that $\phi''(k)=A\alpha (\alpha-1)k^{\alpha-2}<0$ for every $k\in (0,+\infty)$ (see, e.g., \cite[Theorem~4.4]{Rockafellar_1970}). As $\phi(\cdot)$ is continuous at $0$, we can assert that $\phi(\cdot)$ is concave on  $\R_+$. The concavity of $\phi(\cdot)$ on  $\R_+$ is verified similarly. Since both functions $\omega(\cdot)$ and $\phi(\cdot)$ are continuous and concave on $\R_+$, Theorem~\ref{Existence_Thm2} assures the solution existence for the problem \eqref{cost functional_GP_1}--\eqref{state control system_GP_1}.
 \end{proof}
 
 Depending on the displacement of $\alpha$ and $\beta$ on $(0, 1]$, we have four types of the model \eqref{cost functional_GP_1}--\eqref{state control system_GP_1}:
 \begin{itemize}
 	\item ``Linear-linear":  $\phi(k)=Ak$ and  $\omega(c)=c$ (both the per capita production function and the utility function are linear); 
 	\item ``Linear-nonlinear":  $\phi(k)=Ak$ and $\omega(c)=c^{\beta}$ with $\beta\in (0,1)$ (the per capita production function is linear, but the utility function is nonlinear);
 	\item ``Nonlinear-linear":  $\phi(k)=Ak^\alpha$ and $\omega(c)=c$ with $\alpha\in (0,1)$ (the per capita production function is nonlinear, but the utility function is linear);
 		\item ``Nonlinear-nonlinear":  $\phi(k)=Ak^\alpha$ and  $\omega(c)=c^{\beta}$ with $\alpha\in (0,1)$ and $\beta\in (0,1)$ (both the per capita production function and the utility function are nonlinear).
 \end{itemize} 
 
 Although the problem in question of each type has a global solution by Theorem~\ref{Existence_Thm3}, the above classification arranges the difficulties of solving \eqref{cost functional_GP_1}--\eqref{state control system_GP_1}, say, by the Maximum Principle given in \cite[Theorem~9.3.1]{Vinter_2000}. Obviously, problems of the first type are the easiest ones, while those of the fourth type are the most difficult ones.  

\section{Further Discussions}\label{Further discussions}

In this section, first we discuss some assumptions used for getting Theorems \ref{Existence_Thm1} and \ref{Existence_Thm2}. Then we will look deeper into these theorems and the typical optimal economic growth problems in Section~\ref{typical-OEG-problems} by raising some open questions and conjectures about the \textit{uniqueness} and the \textit{regularity} of the global solutions of $(GP)$. 

\subsection{The asymptotic behavior of $\phi$ and its concavity} The results in Section~\ref{general-OEG-problems} were obtained under certain assumptions on the per capita production function $\phi$, which is defined via the production function $F(K,L)$ by the formula 
\begin{equation}\label{phi-def}
\phi (k)=F(k, 1)=\dfrac{F(K,L)}{L}
\end{equation} with $k:=\dfrac{K}{L}$ signifying the capital-to-labor ratio. We want to know: \textit{How the assumptions made on $\phi$ can be traced back to $F$?} 

\begin{proposition}\label{phi_F_pr1}
The per capita production function $\phi:\R_+\to\R_+$ satisfies condition $(c_1)$ if and only if the production function $F:\R^2_+\to\R_+$ has the following property:
\begin{itemize}
\item[$(c_1')$] There exists $c\geq 0$ such that $F(K, L) \leq (c-\sigma)K+cL$ for all $K\geq 0$ and $L>0$.
\end{itemize}
\end{proposition}
\begin{proof} Suppose that $(c_1)$ is satisfied, i.e., there exists $c\geq 0$ such that $\phi (k) \leq (c-\sigma)k+c$ for all  $k\in \R_+$. Then, given any $K\geq 0$ and $L>0$, by substituting $k=\dfrac{K}{L}$ into the last inequality and using \eqref{phi-def}, one gets $$\dfrac{F(K, L)}{L} \leq (c-\sigma)\dfrac{K}{L}+c.$$ This justifies  $(c_1')$. Conversely, suppose that $F(K, L) \leq (c-\sigma)K+cL$ holds for all $K\geq 0$ and $L>0$, where $c\geq 0$ is a constant. Then, letting $L=1$ and $K=k$, where $k\geq 0$ is given arbitrarily, one gets the inequality $\phi (k) \leq (c-\sigma)k+c$. Thus,  $(c_1)$ is fulfilled.
\end{proof}

\begin{proposition}\label{phi_F_pr2}
The function $\phi$ satisfies \eqref{c_2} if and only if $F$ fulfills the following inequality:
\begin{equation}\label{c_2'}
\limsup_{\frac{K}{L}\to +\infty}\dfrac {F(K, L)}{K}<+\infty.
\end{equation}
\end{proposition}
\begin{proof} By \eqref{phi-def}, for any $K>0$ and $L>0$, one has
\begin{equation*}
\dfrac {F(K, L)}{K}=\dfrac {L^{-1}F(K, L)}{L^{-1}K}=\dfrac {\phi(k)}{k}
\end{equation*} with  $k:=\dfrac{K}{L}$. Thus, the equivalence between \eqref{c_2} and \eqref{c_2'} is straightforward.
\end{proof}

Propositions \ref{phi_F_pr1} and \ref{phi_F_pr2} show that the assumption made on $\phi$ in Theorem~\ref{Existence_Thm1} and its equivalent representations given in Proposition~\ref{equiv-conditions} can be checked directly on the original function $F$. 

\begin{proposition}\label{phi_F_pr3} 
		The per capita production function $\phi: \R_+\to \R_+$ is concave on $\R_+$ if and only if the  production function $F:\R^2_+ \to \R_+$ is concave on $\R_+\times (0, +\infty)$.
	\end{proposition} 
	\begin{proof} Firstly, suppose that $F$ is concave on $\R_+\times (0, +\infty)$. Let $k_1, k_2 \in \R_+$ and $\lambda \in [0, 1]$ be given arbitrarily. The concavity of $F$ and \eqref{phi-def} yield
		\begin{equation*}
		F (\lambda (k_1, 1)+(1-\lambda)(k_2, 1))\geq \lambda F(k_1, 1)+(1-\lambda)F(k_2, 1)=\lambda \phi(k_1)+(1-\lambda)\phi(k_2).
		\end{equation*} Since
	$F(\lambda (k_1, 1)+(1-\lambda)(k_2, 1))=F(\lambda k_1+(1-\lambda)k_2, 1)$,
	combining this with \eqref{phi-def}, one obtains $
		\phi (\lambda k_1+(1-\lambda)k_2)\geq \lambda \phi(k_1)+(1-\lambda)\phi(k_2).$ This justifies the concavity of $\phi$.
		
		Now, suppose that $\phi$ is concave on $\R_+$. If $F$ is not concave on $\R_+\times (0, +\infty)$, then there exist $(K_1, L_1), (K_2, L_2)$ in $\R_+\times (0, +\infty)$ and $\lambda \in (0, 1)$ such that 
		\begin{equation*}
		F(\lambda K_1+(1-\lambda)K_2, \lambda L_1+(1-\lambda)L_2) <\lambda F(K_1, L_1)+(1-\lambda)F(K_2, L_2).
		\end{equation*}
		By \eqref{phi-def}, it holds that $F(K,L)=L\phi\big(\dfrac{K}{L}\big)$ for any $(K,L)\in\R_+\times (0, +\infty)$. Therefore, we have
		\begin{equation*}
		[\lambda L_1+(1-\lambda)L_2]\phi\Big (\dfrac{\lambda K_1+(1-\lambda)K_2}{\lambda L_1+(1-\lambda)L_2}\Big)<\lambda L_1\phi\Big (\dfrac{K_1}{L_1}\Big)+(1-\lambda)L_2\phi\Big(\dfrac{K_2}{L_2}\Big).
		\end{equation*}
		Dividing both sides of this inequality by $\lambda L_1+(1-\lambda)L_2$ gives
		\begin{equation}\label{phi-F-ineq1}
		\phi\Big(\dfrac{\lambda K_1+(1-\lambda)K_2}{\lambda L_1+(1-\lambda)L_2}\Big)< \dfrac{\lambda L_1}{\lambda L_1+(1-\lambda)L_2}\phi\Big(\dfrac{K_1}{L_1}\Big)+\dfrac{(1-\lambda)L_2}{\lambda L_1+(1-\lambda)L_2}\phi\Big(\dfrac{K_2}{L_2}\Big).
		\end{equation}
		Setting	$\mu=\dfrac{\lambda L_1}{\lambda L_1+(1-\lambda)L_2}$, one has
		$1-\mu=\dfrac{(1-\lambda)L_2}{\lambda L_1+(1-\lambda)L_2}$, $\mu \in (0,1)$, and $$\mu \dfrac{K_1}{L_1}+(1-\mu)\dfrac{K_2}{L_2}=\dfrac{\lambda K_1+(1-\lambda)K_2}{\lambda L_1+(1-\lambda)L_2}.$$
		Thus, \eqref{phi-F-ineq1} means that 
		\begin{equation*}
		\phi\Big(\mu \dfrac{K_1}{L_1}+(1-\mu)\dfrac{K_2}{L_2}\Big)<\mu \phi\Big(\dfrac{K_1}{L_1}\Big)+(1-\mu)\phi\Big(\dfrac{K_2}{L_2}\Big).
		\end{equation*}
		This contradicts to the assumed concavity of $\phi$ on $\R_+$ and completes the proof.
\end{proof}

We have seen that the assumption on the concavity of $\phi$ used in Theorem~\ref{Existence_Thm2} can be verified directly on $F$.

\subsection{Regularity of the optimal economic growth processes}

Solution regularity is an important concept which helps one to look deeper into the structure of the problem in question. One may have deal with Lipschitz continuity, H\"older continuity, and degree of differentiability of the obtained solutions. We refer to \cite[Chapter~11]{Vinter_2000} for a solution regularity theory in optimal control and to \cite[Theorem~9.2, p.~140]{Kinderlehrer_Stampacchia_1980} for a result on the solution regularity for variational inequalities.

The results of Sections~\ref{general-OEG-problems} and~\ref{typical-OEG-problems} assure that, if some mild assumptions on the per capital function and the utility function are satisfied, then $(GP)$ has a global solution $(\bar k, \bar s)$ with $\bar k(\cdot)$ being absolutely continuous on $[t_0, T]$ and  $\bar s(\cdot)$ being measurable. Since the saving policy $\bar s(\cdot)$ on the time segment $[t_0, T]$ cannot be implemented if it has an infinite number of discontinuities, the following concept of regularity of the solutions of the optimal economic growth problem $(GP)$ appears in a natural way.

\begin{Definition}\label{regularity} {\rm A global solution $(\bar k, \bar s)$ of $(GP)$ is said to be \textit{regular} if the propensity to save function $\bar s(\cdot)$ only has finitely many discontinuities of first type on $[t_0, T]$. This means that there is a positive integer $m$ such that the segment  $[t_0, T]$ can be divided into $m$ subsegments $[\tau_i,\tau_{i+1}]$, $i=0,\dots,m-1$, with $\tau_0=t_0$, $\tau_m=T$, $\tau_i<\tau_{i+1}$ for all $i$,  $\bar s(\cdot)$ is continuous on each open interval $(\tau_i,\tau_{i+1})$, and the one-sided limit $\displaystyle\lim_{t\to {\tau_i}^+}\bar s(t)$ (resp., $\displaystyle\lim_{t\to {\tau_i}^-}\bar s(t)$) exists for each $i\in \{0,1,\dots m-1\}$ (resp., for each $i\in \{1,\dots m\}$).}	
\end{Definition}

In Definition~\ref{regularity}, as $\bar s(t)\in [0,1]$ for every $t\in [t_0, T]$, the one-sided limit $\displaystyle\lim_{t\to {\tau_i}^+}\bar s(t)$ (resp., $\displaystyle\lim_{t\to {\tau_i}^-}\bar s(t)$) must be finite for each $i\in \{0,1,\dots m-1\}$ (resp., for each $i\in \{1,\dots m\}$).

\begin{proposition}\label{regularity-1} 
	Suppose that the function $\phi$ is continuous on $[t_0, T]$. If $(\bar k, \bar s)$ is a regular global solution of $(GP)$, then the capital-to-labor ratio $\bar k(t)$ is a continuous, piecewise continuously differentiable function on the segment  $[t_0, T]$. In particular, the function $\bar k(\cdot)$ is Lipschitz on $[t_0, T]$. 
\end{proposition} 
\begin{proof} Since $(\bar k, \bar s)$ is a regular global solution of $(GP)$, there is a positive integer $m$ such that the segment  $[t_0, T]$ can be divided into $m$ subsegments $[\tau_i,\tau_{i+1}]$, $i=0,\dots,m-1$, and all the requirements stated in Definition~\ref{regularity} are fulfilled. Then, for each $i\in\{0,\dots,m-1\}$, from the first relation in \eqref{state control system_GP} we have
\begin{equation}\label{regularity_1a}\dot {\bar k}(t)=\bar s(t)\phi (\bar k(t))-\sigma \bar k(t),\quad \mbox{a.e.\ } t\in (\tau_i,\tau_{i+1}).\end{equation}  Hence, by the continuity of $\phi$ on $[t_0, T]$ and the continuity of $\bar s(\cdot)$ on $(\tau_i,\tau_{i+1})$, we can assert that the derivative $\dot {\bar k}(t)$ exists for every $t\in (\tau_i,\tau_{i+1})$. Indeed, fixing any point $\bar t\in (\tau_i,\tau_{i+1})$ and using the Lebesgue Theorem \cite[Theorem~6, p.~340]{Kolmogorov_Fomin_1970} for the absolutely continuous function $\bar k(\cdot)$, we have 
\begin{equation}\label{regularity_1b}\bar k(t)=\displaystyle\int_{\bar t}^{t}\dot {\bar k}(\tau)d\tau,\quad \forall  t\in (\tau_i,\tau_{i+1}),\end{equation} where integral on the right-hand-side of the equality is understood in the the Lebesgue sense. Since the Lebesgue integral does not change if one modifies the integrand on a set of zero measure, thanks to~\eqref{regularity_1a} we have
\begin{equation}\label{regularity_1c}\bar k(t)=\displaystyle\int_{\bar t}^{t}[\bar s(\tau)\phi (\bar k(\tau))-\sigma \bar k(\tau)]d\tau.\end{equation} As the integrand of the last integral is a continuous function on $(\tau_i,\tau_{i+1})$, the integration in the Lebesgue sense coincides with that in the Riemanian sense, \eqref{regularity_1c} proves our claim that the derivative $\dot {\bar k}(t)$ exists for every $t\in (\tau_i,\tau_{i+1})$. Moreover, taking derivative of both sides of the equality \eqref{regularity_1b} yields
\begin{equation}\label{regularity_2}
\dot {\bar k}(t)=\bar s(t)\phi (\bar k(t))-\sigma \bar k(t),\quad \forall t\in (\tau_i,\tau_{i+1}).
\end{equation} So, the function $\bar k(\cdot)$ is continuously differentiable of $(\tau_i,\tau_{i+1})$. In addition, the relation~\eqref{regularity_2} and the existence of the finite one-sided limit $\displaystyle\lim_{t\to {\tau_i}^+}\bar s(t)$ (resp., $\displaystyle\lim_{t\to {\tau_i}^-}\bar s(t)$) for each $i$ in $\{0,1,\dots m-1\}$ (resp., for each $i$ in $\{1,\dots m\}$) implies that the one-sided limit $\displaystyle\lim_{t\to {\tau_i}^+}\dot{\bar k}(t)$ (resp., $\displaystyle\lim_{t\to {\tau_i}^-}\dot{\bar k}(t)$) is finite for each $i$ in $\{0,1,\dots m-1\}$ (resp., for each $i$ in $\{1,\dots m\}$. Thus, the restriction of $\bar k(\cdot)$ on each segment $[\tau_i,\tau_{i+1}]$, $i=0,\dots,m-1$, is a continuously differentiable function. We have shown that the capital-to-labor ratio $\bar k(t)$ is a continuous, piecewise continuously differentiable function on the segment  $[t_0, T]$. 

We omit the proof of the Lipschitz property of on $[t_0, T]$ of $\bar k(\cdot)$, which  follows easily from the continuity and piecewise continuously differentiablity of the function  by using the classical mean value theorem.
\end{proof}

We conclude this subsection by two open questions and three independent conjectures, whose solutions or partial solutions will reveal more the beauty of the optimal economic growth model $(GP)$.

\textbf{Open question 1:}\textit{ The assumptions of Theorem~\ref{Existence_Thm1} are not enough to guarantee that $(GP)$ has a regular global solution?}

\textbf{Open question 2:} \textit{The assumptions of Theorem~\ref{Existence_Thm2} are enough to guarantee that every global solution of $(GP)$ is a regular one?}

\textbf{Conjectures:} \textit{The assumptions of Theorem~\ref{Existence_Thm3} guarantee that}
	
	(a) \textit{$(GP)$ has a unique global solution};
	
	(b) \textit{Any global solution of $(GP)$ is a regular one};
	
	(c) \textit{If $(\bar k, \bar s)$ is a regular global solution of $(GP)$, then the optimal propensity to save function $\bar s(\cdot)$ can have at most one discontinuity on the time segment $[t_0, T]$}.
 \begin{acknowledgements}
 This work was supported by National Foundation for Science $\&$ Technology Development (Vietnam) under grant number 101.01-2018.308. The author would like to thank Professor Nguyen Dong Yen for his valuable comments and suggestions on the first version of the present paper.
 \end{acknowledgements}

\end{document}